\documentclass[10pt]{amsart}

\usepackage[english]{my-shortcuts}

\usepackage{a4wide}
\usepackage{amsmath,amsfonts,amssymb,latexsym,amsthm}
\usepackage{comment} 
\usepackage{graphicx}
\usepackage{caption} 

\newtheorem{lemma}{Lemma}

\newtheorem{theorem}{Theorem}

\newcommand{\EE}{\mathbb{E}}

\newcommand{\PP}{\mathbb{P}}
\newcommand{\RR}{\mathbb{R}}

\newcommand{\ve}{\varepsilon}

\begin{document}

\title{An exponentially-averaged Vasyunin formula} \author{S\'ebastien Darses -- Erwan Hillion} 
%\footnote{Aix-Marseille Universit\'e, CNRS, Centrale Marseille, I2M, Marseille, France}

\address{Aix-Marseille Universit\'e, CNRS, Centrale Marseille, I2M, Marseille, France} 
\email{sebastien.darses@univ-amu.fr}

\address{Aix-Marseille Universit\'e, CNRS, Centrale Marseille, I2M, Marseille, France}
\email{erwan.hillion@univ-amu.fr}

\maketitle

\begin{abstract}
We prove a Vasyunin-type formula for an autocorrelation function arising from a Nyman-Beurling criterion generalized to a probabilistic framework. This formula can also be seen as a reciprocity formula for cotangent sums, related to the ones proven in~\cite{BC13},~\cite{ABB17}.
\end{abstract}

\section{Introduction}

The main result of this paper is the following identity:
\begin{theorem} \label{th:Vasyunin} For coprime $m,n\ge 1$, we have
\begin{eqnarray*}
mn\int_0^\infty \left(\frac{1}{mt}-\frac{1}{e^{mt}-1}\right)\left(\frac{1}{nt}-\frac{1}{e^{nt}-1}\right) dt     &=&  -\frac{1}{2} + \cal C(n+m) + \frac{m-n}{2}\log\left(\frac{n}{m}\right) \\
&& - \frac{\pi}{2} \sum_{k=1}^{n-1}\frac{m k}{n} \cot\left(\frac{m k \pi}{n}\right) - \frac{\pi}{2} \sum_{l=1}^{m-1} \frac{n l}{m} \cot\left(\frac{n l \pi}{m}\right) ,
\end{eqnarray*}
where
\begin{equation*} 
\cal C = 1-\int_0^1 \left(\frac{1}{t(e^t-1)} - \frac{1}{t^2}+\frac{1}{2t} \right)dt - \int_1^\infty \frac{dt}{t(e^t-1)}=\frac 12(\log 2\pi -\gamma).
\end{equation*}
\end{theorem}

Theorem~\ref{th:Vasyunin} was obtained when studying a probabilistic version of the Nyman-Beurling criterion for the Riemann hypothesis (RH). One of the main results of the deterministic Nyman-Beurling approach for RH (see~\cite{Nym50,Beu55}), improved by B\'aez-Duarte \textit{et al.} in~\cite{BD03},~\cite{BDBLS00}, is the following:

\begin{theorem}
In $H=L^2(0,\infty)$, set $\chi : t \mapsto 1_{]0,1]}(t)$ and $\rho_n : t \mapsto \left\{\frac{1}{nt}\right\}$, $n \ge 1$.  Then RH holds if and only if
\bean
d_N = d_{H}\left(\chi,{\rm Span}(\rho_1,\ldots, \rho_N)\right) \xrightarrow[N \rightarrow \infty]{} 0.
\eean
\end{theorem}

Here, $\{\cdot \}$ denote the fractional part, and the notation $d_{\mathcal{F}}(f,F)$ stands for the distance between the vector $f$ and the subspace $F$ in the Hilbert space $\mathcal{F}$.\\

The squared distance $d_N^2$ can be expressed as a quotient of Gram determinants: 
\bean
d_N^2 = \frac{\det({\rm Gram}(\chi,\rho_1,\ldots,\rho_N))}{\det({\rm Gram}(\rho_1,\ldots,\rho_N))}.
\eean
The computation of the coefficients of these Gram matrices is related to the study of the autocorrelation function
\bean
\lambda \mapsto A(\lambda) = \int_0^\infty \left\{\frac{1}{t}\right\}\left\{\frac{1}{\lambda t}\right\} dt.
\eean
Indeed, for every $n,m \ge 1$, a simple change of variables gives
\bean
\langle \rho_n , \rho_m \rangle = \int_0^\infty \left\{\frac{1}{nt}\right\}\left\{\frac{1}{mt}\right\} dt = \frac{1}{n} A\left(\frac{m}{n}\right) = \frac{1}{m} A\left(\frac{n}{m}\right).
\eean
\medskip

The autocorrelation function $A(\lambda)$ has been studied in~\cite{BDBLS05}, where the authors prove in particular that $A$ is non-differentiable at each rational point. One of the most useful technical tool for the study of $A$ is the Vasyunin formula \cite{Vas95},~\cite[p.141]{BDBLS00}:
\begin{eqnarray*}
m n \int_0^\infty \left\{\frac{1}{nt}\right\}\left\{\frac{1}{mt}\right\} dt   &=&  \mathcal{C}(n+m) + \frac{m-n}{2}\log\left(\frac{n}{m}\right) \\
&& - \frac{\pi}{2} \sum_{k=1}^{n-1}\left\{\frac{m k}{n}\right\} \cot\left(\frac{k \pi}{n}\right) - \frac{\pi}{2} \sum_{l=1}^{m-1} \left\{\frac{l n}{m}\right\} \cot\left(\frac{l \pi}{m}\right),
\end{eqnarray*}
for the same constant $ \mathcal{C} = \frac 12(\log 2\pi -\gamma) $.

\medskip

The similarity between Theorem~\ref{th:Vasyunin} and Vasyunin formula is striking. In Section~\ref{sec:pNB}, we explain how the left-hand side in Theorem~\ref{th:Vasyunin} can be seen as an "exponentially-averaged" autocorrelation function coming from a probabilistic Nyman-Beurling criterion. More precisely, it can be written $m \mathcal{A}\left(\frac{m}{n}\right)$, or $n \mathcal{A}\left(\frac{n}{m}\right)$, where 
\bean
\mathcal{A}(\lambda) = \int_0^\infty \left(\frac{1}{t}-\frac{1}{e^{t}-1}\right) \left(\frac{1}{\lambda t}-\frac{1}{e^{\lambda t}-1}\right) dt.
\eean

\bigskip

Vasyunin's formula is one of the motivations for the study of cotangent sums, which has been an area of active research these past years, see for instance the introduction of~\cite{BC13}, or the references ~\cite{MR16},~\cite{Bet15}.

\medskip

A remarkable property of cotangent sums has been unveiled by Bettin and Conrey in~\cite{BC13}, who obtained a so-called {\it reciprocity formula}. Following their notations, we set 
\bean 
c(x)=-\sum_{a=1}^{k-1} \frac{a}{k} \cot\left(\frac{\pi a h}{k} \right),
\eean
where $x=h/k$, $k>0$ and ${\rm gcd}(h,k)=1$. The reciprocity formula states that 
\begin{equation*}
    xc(x)+c\left(\frac{1}{x}\right)-\frac{1}{\pi k}=g(x),
\end{equation*} 
for a smooth function $g$ defined from Eisenstein series. What makes the reciprocity formula interesting is the contrast with the fact that the function $c$, defined on the set of rational numbers, cannot be extended into a continuous function on $\RR_+^*$.

\medskip

Theorem~\ref{th:Vasyunin} can be seen as another formulation of the reciprocity formula for cotangent sums. Indeed, with the notations of~\cite{BC13}, Theorem~\ref{th:Vasyunin} can be rewritten as
\begin{equation*}
    xc(x)+c\left(\frac{1}{x}\right)-\frac{1}{\pi k} = \frac{1}{\pi} \left(2 x \mathcal{A}(x)-2(1+x)\cal C+(x-1) \log(x) \right). 
\end{equation*}
Both results combined give a simple representation formula for the function $g$. 

\medskip

Actually, the reciprocity formula is stated and proved in~\cite{BC13} for more general cotangent sums. More precisely, given $a>0$, let us consider the arithmetic sum
\bean
c_a\left(\frac{h}{k}\right) = k^a \sum_{k=1}^{m-1} \cot\left(\frac{\pi m h}{k}\right) \zeta\left(-a,\frac{m}{k}\right),
\eean
where $\zeta(a,x)$ denotes the Hurwitz zeta function. For such sums, Bettin and Conrey proved that the function
\bean 
g_a\left(\frac{h}{k}\right) = c_a\left(\frac{h}{k}\right)-\left(\frac{k}{h}\right)^{1+a} c_a\left(-\frac{k}{h}\right) + \frac{k^a a \zeta(1-a)}{\pi h}
\eean
can be extended from $\mathbb{Q}$ to an analytic function on $\mathbb{C}-\RR_{\leq 0}$.

\medskip

In the article~\cite{ABB17} by Auli, Bayad and Beck, the authors give a more explicit expression for the function $g_a$, written as a contour integral.

\section{An exponentially-averaged autocorrelation function} \label{sec:pNB}

Although they are not needed for the proof of Theorem~\ref{th:Vasyunin}, the results stated in this paragraph explain how the function $\mathcal{A}$ arises from a Nyman-Beurling criterion set in a probabilistic framework, and how the study of reciprocity formulas for cotangent sums is related to RH, see~\cite{DH18}.\\

Let $(\Omega,\mathcal{F},\mathbb{P})$ be a probability space. We consider the Hilbert space $\mathcal{H} = L^2(\Omega \times [0,\infty))$ endowed with the scalar product
\bean
\langle X,Y \rangle_{\mathcal{H}} = \int_0^\infty \int_\Omega X(\omega,t) Y(\omega,t) d\PP(\omega) dt.
\eean
Let $(X_k)_{k \ge 1}$ be an independent sequence of random variables, defined on $(\Omega,\mathcal{F},\mathbb{P})$,  with exponential distribution $X_k \sim \mathcal{E}(1)$. For each $k \ge 1$, the mapping $R_k : (\omega,t) \mapsto \left\{\frac{X_k(\omega)}{k t}\right\}$ belongs to $\mathcal{H}$. We also define $\chi\in \mathcal{H}$ by $\chi(\omega,t) = 1_{[0,1]}(t)$.\\

In~\cite{DH18}, we prove the following:
\begin{theorem}\label{prop:pNB}
If $D_N = d_{\mathcal{H}}\left(\chi,{\rm Vect}(R_1,\ldots, R_N)\right) \xrightarrow[N \rightarrow \infty]{} 0$, then RH holds.
\end{theorem}

As in the deterministic case, the squared distance $D_N^2$ can be expressed as a quotient of Gram determinants, which is a motivation for the computation of the scalar products $\langle R_n, R_m \rangle_{\mathcal{H}}$, for $m,n \geq 1$. In order to compute these scalar products, we use the following fact: if $Z$ is a random variable with exponential distribution $\mathcal{E}(\alpha)$, $\alpha >0$, then
\bean
\EE\left[\left\{Z\right\}\right] = \frac{1}{\alpha}-\frac{1}{e^{\alpha}-1},
\eean
where $\EE[X]$ is the expectation of the random variable $X$. Such a formula is obtained by straightforward calculations.\\

As $\frac{X_n}{nt}$ (resp. $\frac{X_m}{mt}$) follows the exponential distribution $\mathcal{E}(nt)$ (resp. $\mathcal{E}(mt)$) we obtain, by independence of $R_n$ with $R_m$:
\bean
\langle R_n, R_m \rangle_{\mathcal{H}} = \int_0^\infty \left(\frac{1}{mt}-\frac{1}{e^{mt}-1}\right)\left(\frac{1}{nt}-\frac{1}{e^{nt}-1}\right) dt = \frac{1}{n} \mathcal{A}\left(\frac{m}{n}\right) = \frac{1}{m} \mathcal{A}\left(\frac{n}{m}\right).
\eean

This formula explains the terminology "exponentially-averaged autocorrelation function" for $\mathcal{A}(\lambda)$ and "exponentially-averaged Vasyunin formula" for Theorem~\ref{th:Vasyunin}.\\

We draw the graph of $\lambda\mapsto \mathcal{A}(1/\lambda)$, which has been obtained using Theorem~\ref{th:Vasyunin}.  This graph can be compared with the one of the deterministic autocorrelation function $A$ as seen in~\cite{BDBLS05}. Although there are some global similarities in the behaviour of both functions, the function $\mathcal{A}$ is analytic whereas $A$ is not differentiable (cf Eq.(2) in \cite{BDBLS05}).

\begin{center}
\includegraphics[width=12cm,height=8cm]{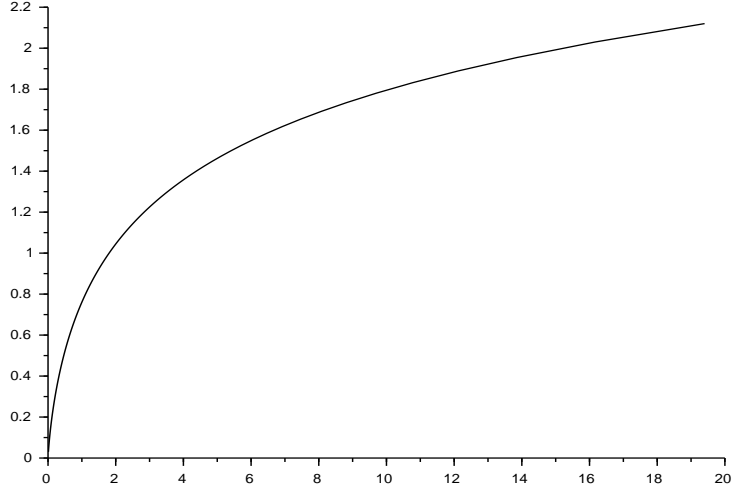} \label{fig:autoco}
\captionof{figure}{The exponentially-averaged autocorrelation function $\lambda\mapsto \mathcal{A}(1/\lambda)$, $\lb>0$.}
\end{center}

\section{Computation of $\mathcal{C}$}

Numerical evidence suggested that $\cal C = \frac{1}{2}(\log(2\pi)-\gamma)$, where $\gamma$ is the Euler constant. Balazard actually proved this identity \cite{Bal18}. We thank Michel Balazard for his proof and for authorizing us to reproduce it below.

\begin{prop}[\cite{Bal18}]\label{bal18}
One has
\[
1-\int_0^1 \left (\frac{1}{t(e^t-1)}-\frac 1{t^2}+\frac 1{2t}\right) dt-\int_1^{\infty} \frac{dt}{t(e^t-1)}=\frac 12(\log 2\pi -\gamma).
\]
\end{prop}

\begin{proof}(translated from \cite{Bal18})
One has
\begin{multline*}
\int_0^1 \left (\frac{1}{t(e^t-1)}-\frac 1{t^2}+\frac 1{2t}\right) dt+\int_1^{\infty} \frac{dt}{t(e^t-1)}=\\
\lim_{x \rightarrow 0}\left(\int_0^1 \left (\frac{1}{e^t-1}-\frac 1{t}+\frac 1{2}\right) e^{-tx}\frac{dt}t+\int_1^{\infty} \frac{e^{-tx}}{e^t-1}  \frac{dt}t\right).
\end{multline*}

But, for $x>0$,
\bean
I(x) & = & \int_0^1 \left (\frac{1}{e^t-1}-\frac 1{t}+\frac 1{2}\right) e^{-tx}\frac{dt}t+\int_1^{\infty} \frac{e^{-tx}}{e^t-1}  \frac{dt}t\\
    & = & \int_0^{\infty} \left (\frac{1}{e^t-1}-\frac 1{t}+\frac 1{2}\right) e^{-tx}\frac{dt}t +\int_1^{\infty} \left(\frac 1{t}-\frac 1{2}\right)   e^{-tx} \frac{dt}t\cdotp
\eean

On one hand,
\[
\int_0^{\infty} \left (\frac{1}{e^t-1}-\frac 1{t}+\frac 1{2}\right) e^{-tx}\frac{dt}t =\log \Gamma(x)-(x-1/2)\log x +x-\frac 12 \log 2\pi,
\]
(Binet, 1839, cf. \cite{WW27}, \S $\mathbf {12\cdot 31}$, p. 249).

On the other hand, as $x$ tends to $0$,
\begin{equation*}
\int_1^{\infty} \left(\frac 1{t}-\frac 1{2}\right)   e^{-tx} \frac{dt}t =-\frac 12\int_1^{\infty}  e^{-tx} \frac{dt}t+1 +o(1)
\end{equation*}
with
\begin{eqnarray*}
\int_1^{\infty}   e^{-tx} \frac{dt}t &= & \int_x^{\infty}   e^{-t} \frac{dt}t\\
&=&\int_x^1   (e^{-t}-1) \frac{dt}t-\log x+ \int_1^{\infty}   e^{-t} \frac{dt}t\\
&=& -\log x -\gamma +o(1),
\end{eqnarray*}
by virtue of a classical formula for $\gamma$ (cf. \cite{WW27}, \S $\mathbf {12\cdot 2}$, Example 4, p. 243).

Finally,
\begin{eqnarray*}
I(x)&=&\log \Gamma(x)-(x-1/2)\log x +x-\frac 12 \log 2\pi -\frac 12\big(-\log x -\gamma +o(1)\big) +1+o(1)\\
&=&1+\frac 12(\gamma -\log 2\pi)+o(1),
\end{eqnarray*}
where one used $x\Gamma(x) = \Gamma(x+1)$, $x>0$, and $\Gamma(1)=1$. This completes the proof.
\end{proof}

\section{Proof of Theorem~\ref{th:Vasyunin}}

We need several technical lemmas to proceed the proof of Theorem~\ref{th:Vasyunin}.
\begin{lemma}
We have the following expansions
\begin{equation} \label{eq:FR1}
\frac{1}{z^n-1}-\frac{1}{n}\frac{1}{z-1} = \frac{1}{2n} \sum_{k=1}^{n-1} \frac{2 \cos\left(\frac{2k\pi}{n}\right) z-2}{z^2-2\cos\left(\frac{2k\pi}{n}\right)z+1},
\end{equation}
\begin{equation} \label{eq:FR2}
\frac{z^{n-1}}{z^n-1}-\frac{1}{n}\frac{1}{z-1} = \frac{1}{2n} \sum_{k=1}^{n-1} \frac{-2 \cos\left(\frac{2k\pi}{n}\right)+2z}{z^2-2\cos\left(\frac{2k\pi}{n}\right)z+1},
\end{equation}
\begin{equation} \label{eq:FR3}
\frac{z^{n}+1}{z^n-1}-\frac{1}{n}\frac{z+1}{z-1} = \frac{1}{n} \frac{z^2-1}{2z} \sum_{k=1}^{n-1} \frac{1}{\frac{z^2+1}{2z}-\cos\left(\frac{2k\pi}{n}\right)}.
\end{equation}
\end{lemma}

\begin{proof} We consider the $n$-th roots of unity $\omega_{k,n} = e^{\frac{2i\pi k}{n}}$ for $k \in \{0,\ldots,n-1\}$. We have
\begin{equation*}
\frac{1}{z^n-1} = \frac{1}{\prod_{k=0}^{n-1} (z-\omega_{k,n})} = \frac{1}{n} \sum_{k=0}^{n-1} \frac{\omega_{k,n}}{z-\omega_{k,n}} = \frac{1}{n} \frac{1}{z-1}+\frac{1}{n} \sum_{k=1}^{n-1} \frac{\omega_{k,n}}{z-\omega_{k,n}}.
\end{equation*}
We sum in both directions to obtain
\begin{equation*}
\frac{1}{z^n-1}-\frac{1}{n}\frac{1}{z-1} = \frac{1}{2n} \sum_{k=1}^{n-1}  \frac{\omega_{k,n}}{z-\omega_{k,n}}+ \frac{\omega_{n-k,n}}{z-\omega_{n-k,n}} = \frac{1}{2n} \sum_{k=1}^{n-1} \frac{2 \cos\left(\frac{2k\pi}{n}\right) z-2}{z^2-2\cos\left(\frac{2k\pi}{n}\right)z+1},
\end{equation*}
which is exactly Equation~\eqref{eq:FR1}. 

\medskip

In order to obtain the second identity, we write
\begin{equation*}
\frac{1}{z-1}+\sum_{k=1}^{n-1} \frac{1}{z-\omega_{k,n}} = \sum_{k=0}^{n-1} \frac{1}{z-\omega_{k,n}} = \frac{P'(z)}{P(z)},
\end{equation*}
where $P(z) = \prod_{k=0}^{n-1} (z-\omega_{k,n}) = z^n-1$. We then have
\begin{equation*}
\frac{nz^{n-1}}{z^n-1}-\frac{1}{z-1}=\frac{1}{2} \sum_{k=1}^{n-1} \frac{1}{z-\omega_{k,n}}+ \frac{1}{z-\omega_{n-k,n}} = \frac{1}{2} \sum_{k=1}^{n-1} \frac{-2 \cos\left(\frac{2k\pi}{n}\right)+2z}{z^2-2\cos\left(\frac{2k\pi}{n}\right)z+1}.
\end{equation*}

In order to obtain~\eqref{eq:FR3}, we notice that
\begin{equation*}
\frac{z^{n}+1}{z^n-1}-\frac{1}{n}\frac{z+1}{z-1} = z \left(\frac{z^{n-1}}{z^n-1}-\frac{1}{n}\frac{1}{z-1}\right) + \left(\frac{1}{z^n-1}-\frac{1}{n}\frac{1}{z-1}\right),
\end{equation*}
and we use Equations~\eqref{eq:FR1} and~\eqref{eq:FR2}.
\end{proof}

\begin{lemma} \label{lem:suminvcos} For any $n \ge 1$ and $a \notin \frac{2\pi}{n} \mathbb{Z}$, we have
\begin{equation*}
\frac{1}{n} \sum_{k=1}^{n-1} \frac{1}{\cos(a)-\cos\left(\frac{2k\pi}{n}\right)} = \frac{1}{\sin(a)} \left( \frac{1}{n} \cot\left(\frac{a}{2}\right) - \cot\left(\frac{n a}{2}\right)\right).
\end{equation*}
\end{lemma}

\begin{proof}
We evaluate equation~\eqref{eq:FR3} at $z=e^{i a}$:
\begin{equation*}
\frac{e^{ian}+1}{e^{ian}-1}-\frac{1}{n}\frac{e^{ia}+1}{e^{ia}-1} = \frac{1}{n} \frac{e^{2ia}-1}{2e^{ia}} \sum_{k=1}^{n-1} \frac{1}{\frac{e^{2ia}+1}{2e^{i \alpha}}-\cos\left(\frac{2k\pi}{n}\right)}.
\end{equation*}
The right-hand side reads
\begin{equation*}
\frac{1}{n} \frac{e^{2ia}-1}{2e^{ia}} \sum_{k=1}^{n-1} \frac{1}{\frac{e^{2ia}+1}{2e^{i \alpha}}-\cos\left(\frac{2k\pi}{n}\right)} = \frac{i}{n} \sin(a) \sum_{k=1}^{n-1} \frac{1}{\cos(a)-\cos\left(\frac{2k\pi}{n}\right)},
\end{equation*}
and the left-hand side:
\begin{equation*}
\frac{e^{ian}+1}{e^{ian}-1}-\frac{1}{n}\frac{e^{ia}+1}{e^{ia}-1} = \frac{2 \cos(na/2)}{2 i \sin(na/2)}- \frac{1}{n} \frac{2 \cos(a/2)}{2 i \sin(a/2)} = -i \left(\cot\left(\frac{n a}{2}\right)- \frac{1}{n} \cot\left(\frac{a}{2}\right) \right).
\end{equation*}
We then obtain the desired result. 
\end{proof}

\begin{lemma} \label{lem:ArctanInt}
For any $a \in (0,2\pi), a \neq \pi$,
\begin{equation*}
    \int_1^\infty \frac{1}{z^2-2\cos(a)z+1} dz = \frac{1}{\sin(a)} \left(\frac{\pi}{2}-\frac{a}{2}\right).
\end{equation*}
\end{lemma}

\begin{proof} 
We have, for $|\alpha| < 1$,
\begin{equation*}
\int_1^\infty \frac{1}{z^2-2\alpha z+1} dz = \frac{1}{\sqrt{1-\alpha^2}} \arctan\left(\sqrt{\frac{1+\alpha}{1-\alpha}}\right),
\end{equation*}
which is obtained using the change of variables $x= \frac{z-\alpha}{\sqrt{1-\alpha^2}}$ and $ \pi/2-\arctan(x)= \arctan(1/x)$. If $\alpha = \cos(a)$ for some $a \in (0,2\pi)$, $a \neq \pi$, the above expression simplifies into
\begin{eqnarray*}
\frac{1}{\sqrt{1-\alpha^2}} \arctan\left(\sqrt{\frac{1+\alpha}{1-\alpha}}\right) &=& \frac{1}{|\sin(a)|}\arctan\left(\sqrt{\frac{2 \cos(a/2)^2}{2 \sin(a/2)^2}}\right) \\
&=& \frac{1}{|\sin(a)|} \arctan\left(\left|\frac{1}{\tan(a/2)}\right|\right) \\
&=& \frac{1}{|\sin(a)|} \left| \frac{\pi}{2}-\frac{a}{2}\right| = \frac{1}{\sin(a)} \left(\frac{\pi}{2}-\frac{a}{2}\right),
\end{eqnarray*}
the last identity being obtained by studying separately the cases $0 < a < \pi$ and $\pi < a < 2\pi$. \end{proof}

\medskip

For $m,n \ge 1$, we set 
\begin{equation*}
I(m,n) = \int_0^\infty \left(\frac{1}{e^{nx}-1}-\frac{1}{n}\frac{1}{e^x-1}\right) \left( \frac{1}{e^{mx}-1}-\frac{1}{m}\frac{1}{e^x-1}\right) dx.
\end{equation*}

The change of variables $z=e^x$ gives
\begin{equation*}
I(m,n) = \int_1^\infty \left(\frac{1}{z^n-1}-\frac{1}{n}\frac{1}{z-1}\right) \left( \frac{1}{z^m-1}-\frac{1}{m}\frac{1}{z-1}\right) \frac{dz}{z}.
\end{equation*}

\begin{lemma} \label{prop:ImnCoprime}
Let $m,n \ge 2$ be coprime numbers. Then 
\begin{eqnarray*}
I(m,n) &=& \frac{m-1}{2m} \frac{\log(n)}{n} + \frac{n-1}{2n} \frac{\log(m)}{m} \\
&& + \frac{1}{2 n}\sum_{k=1}^{n-1} \left( \cot\left(\frac{mk\pi}{n}\right)-\frac{1}{m} \cot\left(\frac{k \pi}{n}\right)  \right) \left(\frac{\pi}{2}-\frac{k \pi}{n}\right) \\
&&  +\frac{1}{2m}\sum_{l=1}^{m-1} \left(\cot\left(\frac{nl\pi}{m}\right)- \frac{1}{n} \cot\left(\frac{l \pi}{m}\right) \right) \left(\frac{\pi}{2}-\frac{l \pi}{m}\right).
\end{eqnarray*}
\end{lemma}

\begin{proof} 
We deduce from equation~\eqref{eq:FR1} that
\begin{equation*}
I(m,n) = \frac{1}{nm} \sum_{k=1}^{n-1} \sum_{l=1}^{m-1} \int_1^{\infty} \frac{\cos\left(\frac{2k\pi}{n}\right) z-1}{z^2-2\cos\left(\frac{2k\pi}{n}\right)z+1} \frac{\cos\left(\frac{2l\pi}{m}\right) z-1}{z^2-2\cos\left(\frac{2l\pi}{m}\right)z+1} \frac{1}{z} dz.
\end{equation*}
For $k \in \{1,\ldots,n-1\}$ and $l \in \{1,\ldots,m-1\}$, we set $\alpha_k = \cos\left(\frac{2k\pi}{n}\right)$, $\beta_l = \cos\left(\frac{2l\pi}{m}\right)$, and
\begin{equation*}
F_{k,l}(z) = \frac{\alpha_k z-1}{z^2-2\alpha_k z+1} \frac{\beta_l z-1}{z^2-2\beta_l z+1} \frac{1}{z}.
\end{equation*}
As $m$ and $n$ are coprime numbers, we have $\alpha_k \neq \beta_l$. Therefore,
\begin{eqnarray*}
F_{k,l}(z) &=& \frac{1}{z} + \frac{\frac{-1}{2}z+\frac{1-2\alpha_k^2 + \alpha_k \beta_l}{2 (\beta_l-\alpha_k)}}{z^2-2\alpha_k z +1} + \frac{\frac{-1}{2}z-\frac{1-2\beta_l^2 + \alpha_k \beta_l}{2 (\beta_l-\alpha_k)}}{z^2-2\beta_l z +1} \\
&=& \frac{1}{z}-\frac{1}{4} \frac{2z-2\alpha_k}{z^2-2\alpha_k z+1}+\frac{1}{2} \frac{1-\alpha_k^2}{\beta_l-\alpha_k} \frac{1}{z^2-2\alpha_k z+1}\\
& & -\frac{1}{4} \frac{2z-2\beta_l}{z^2-2\beta_l z+1}+\frac{1}{2} \frac{1-\beta_l^2}{\alpha_k-\beta_l} \frac{1}{z^2-2\beta_l z+1}.
\end{eqnarray*}
We use equation~\eqref{eq:FR2} to write
\begin{eqnarray*}
\sum_{k=1}^{n-1} \sum_{l=1}^{m-1} \frac{1}{z}-\frac{1}{4} \frac{2z-2\alpha_k}{z^2-2\alpha_k z+1}-\frac{1}{4} \frac{2z-2\beta_l}{z^2-2\beta_l z+1} &=& \frac{(m-1)(n-1)}{z} - \frac{(m-1)n}{2} \left(\frac{z^{n-1}}{z^n-1}-\frac{1}{n}\frac{1}{z-1}\right) \\&& - \frac{(n-1)m}{2} \left(\frac{z^{m-1}}{z^m-1}-\frac{1}{m}\frac{1}{z-1}\right)\\
&=& \frac{1}{2} \frac{\partial}{\partial z} \log\left(\frac{z^{2(m-1)(n-1)}(z-1)^{m-1+n-1}}{(z^n-1)^{m-1} (z^m-1)^{n-1}}\right).
\end{eqnarray*}
From the limits
\bean
    \lim_{z \rightarrow \infty} \log\left(\frac{z^{2(m-1)(n-1)}(z-1)^{m-1+n-1}}{(z^n-1)^{m-1} (z^m-1)^{n-1}}\right) & = & 0, \\
    \lim_{z \rightarrow 1} \log\left(\frac{z^{2(m-1)(n-1)}(z-1)^{m-1+n-1}}{(z^n-1)^{m-1} (z^m-1)^{n-1}}\right) & = &  \log\left(\frac{1}{n^{m-1} m^{n-1}} \right),
\eean
we deduce
\begin{equation}\label{eq:Imn1}
\frac{1}{mn} \sum_{k=1}^{n-1} \sum_{l=1}^{m-1} \int_1^\infty \frac{1}{z}-\frac{1}{4} \frac{2z-2\alpha_k}{z^2-2\alpha_k z+1}-\frac{1}{4} \frac{2z-2\beta_l}{z^2-2\beta_l z+1} dz = \frac{m-1}{2m} \frac{\log(n)}{n} + \frac{n-1}{2n} \frac{\log(m)}{m}.
\end{equation}

In order to compute the integral of the remaining terms, we first use Lemma~\ref{lem:suminvcos} to write
\begin{eqnarray*}
\frac{1}{mn} \sum_{k=1}^{n-1} \sum_{l=1}^{m-1} \frac{1-\alpha_k^2}{\beta_l-\alpha_k} \frac{1}{z^2-2\alpha_k z+1} &=& - \frac{1}{n}\sum_{k=1}^{n-1}\left( \frac{1}{m} \sum_{l=1}^{m-1} \frac{1}{\cos\left(\frac{2 k \pi}{n}\right)-\cos\left(\frac{2 l \pi}{m}\right)} \right)\frac{\sin^2\left(\frac{2 k \pi}{n}\right)}{z^2-2\cos\left(\frac{2 k \pi}{n}\right) z+1} \\
&=& - \frac{1}{n}\sum_{k=1}^{n-1}\frac{\frac{1}{m} \cot\left(\frac{k \pi}{n}\right) - \cot\left(\frac{mk\pi}{n}\right)}{\sin\left(\frac{2 k \pi}{n}\right)} \frac{\sin^2\left(\frac{2 k \pi}{n}\right)}{z^2-2\cos\left(\frac{2 k \pi}{n}\right) z+1} \\
&=& - \frac{1}{n}\sum_{k=1}^{n-1} \left( \frac{1}{m} \cot\left(\frac{k \pi}{n}\right) - \cot\left(\frac{mk\pi}{n}\right)\right) \frac{\sin\left(\frac{2 k \pi}{n}\right)}{z^2-2\cos\left(\frac{2 k \pi}{n}\right) z+1}.
\end{eqnarray*}
Lemma~\ref{lem:ArctanInt} yields
\begin{equation*}
    \int_1^\infty \frac{\sin\left(\frac{2 k \pi}{n}\right)}{z^2-2\cos\left(\frac{2 k \pi}{n}\right) z+1} dz = \left(\frac{\pi}{2}-\frac{k \pi}{n}\right).
\end{equation*}

\medskip

We thus have shown that
\begin{equation}\label{eq:Imn2}
\int_1^\infty \frac{1}{mn} \sum_{k=1}^{n-1} \sum_{l=1}^{m-1} \frac{1-\alpha_k^2}{\beta_l-\alpha_k} \frac{1}{z^2-2\alpha_k z+1} dz = - \frac{1}{n}\sum_{k=1}^{n-1} \left( \frac{1}{m} \cot\left(\frac{k \pi}{n}\right) - \cot\left(\frac{mk\pi}{n}\right)\right) \left(\frac{\pi}{2}-\frac{k \pi}{n}\right).
\end{equation}
Similarly,
\begin{equation} \label{eq:Imn3}
\int_1^\infty \frac{1}{mn} \sum_{k=1}^{n-1} \sum_{l=1}^{m-1} \frac{1-\beta_l^2}{\alpha_k-\beta_l} \frac{1}{z^2-2\beta_l z+1} dz = - \frac{1}{m}\sum_{l=1}^{m-1} \left( \frac{1}{n} \cot\left(\frac{l \pi}{m}\right) - \cot\left(\frac{nk\pi}{m}\right)\right) \left(\frac{\pi}{2}-\frac{l \pi}{m}\right).
\end{equation}
Lemma~\ref{prop:ImnCoprime} is then proven by summing the identities~\eqref{eq:Imn1},~\eqref{eq:Imn2} and~\eqref{eq:Imn3}.
\end{proof}

\begin{lemma} \label{prop:I1n}
Let $n \ge 1$. There exists $\cal C>0$, which does not depend on $n$, such that
\begin{equation*}
\int_0^\infty \left(\frac{1}{t}-\frac{1}{e^{t}-1}\right)\left(\frac{1}{nt}-\frac{1}{e^{nt}-1}\right) dt = \cal C\left(1+\frac{1}{n}\right) -\frac{1}{2n} - \frac{n-1}{2n} \log(n) + \frac{1}{2n} \sum_{k=1}^{n-1} \cot\left(\frac{k \pi}{n}\right)\left( \frac{\pi}{2}-\frac{k \pi}{n}\right).
\end{equation*}
\end{lemma}

\begin{proof} It suffices to compute the integrals
\begin{equation*}
I_1 = \int_\ve^{\infty} \frac{dt}{nt^2} \ , \ I_2 = -\int_\ve^{\infty} \frac{dt}{t (e^{nt}-1)} \ , \ I_3 = -\int_\ve^{\infty} \frac{dt}{nt(e^t-1)} \ , \ I_4 = \int_\ve^{\infty} \frac{1}{e^{t}-1}\frac{1}{e^{nt}-1} dt,
\end{equation*}
to sum them up and let $\ve \rightarrow 0$.

\medskip

We have $I_1 = \frac{1}{n \ve}$. In order to compute $I_2$ and $I_3$, we notice that the function
\begin{equation*}
t \longmapsto \frac{1}{t(e^t-1)} - \frac{1}{t^2}+\frac{1}{2t}
\end{equation*}
can be continuously extended at $t=0$, so we have
\begin{equation*}
\int_\ve^\infty\frac{dt}{t(e^t-1)} = \frac{1}{\ve}+\frac{1}{2}\log(\ve) - \cal C + o(\ve).
\end{equation*}
Recall that 
\begin{equation*} 
\cal C = 1-\int_0^1 \left(\frac{1}{t(e^t-1)} - \frac{1}{t^2}+\frac{1}{2t} \right)dt - \int_1^\infty \frac{dt}{t(e^t-1)}.
\end{equation*}
Hence,
\begin{equation*}
I_2 = -\frac{1}{n\ve}-\frac{1}{2} \log(n\ve) + \cal C + o(\ve) , 
\end{equation*}
\begin{equation*}
I_3 = -\frac{1}{n\ve}-\frac{1}{2n} \log(\ve) + \frac{\cal C}{n} + o(\ve). 
\end{equation*}

It remains to compute $I_4$. The change of variables $x=e^t$ gives, with $c = \exp(\ve)$,
\begin{equation*}
I_4 = \int_c^\infty \frac{1}{x} \frac{1}{x-1} \frac{1}{x^n-1} dx.
\end{equation*}
Still using $\omega_{k,n} = e^{\frac{2i \pi k}{n}}$ and $\alpha_k = \cos\left(\frac{2 k \pi}{n}\right)$, we perform the partial fraction expansion of $\frac{1}{x^n-1}$:
\begin{equation*}
\frac{n}{x^n-1} = \sum_{k=0}^{n-1}\frac{\omega_{k,n}}{x-\omega_{k,n}}
=\frac{1}{x-1}+ \sum_{k=1}^{n-1} \frac{x \alpha_k-1}{x^2-2 \alpha_k x+1}.
\end{equation*}
We compute
\begin{eqnarray*}
\int_c^\infty \frac{dx}{x(x-1)^2} &=& \int_c^\infty \left(\frac{1}{x}-\frac{1}{x-1}+\frac{1}{(x-1)^2} \right) dx \\
&=& \left[\log\left(\frac{x}{x-1}\right)\right]_c^{\infty}+\left[-\frac{1}{x-1}\right]_{c}^{\infty} \\
&=& \log\left(\frac{c-1}{c}\right)+\frac{1}{c-1} \\
&=& \frac{1}{\ve} + \log(\ve) -\frac{1}{2} + o_{\ve \rightarrow 0}(1).
\end{eqnarray*}

For $1 \le k \le n-1$, we have
\begin{equation*}
\frac{1}{x} \frac{1}{x-1}  \frac{x \alpha_k-1}{x^2-2 \alpha_k x+1} = \frac{1}{x}-\frac{1}{2}\frac{1}{x-1}-\frac{1}{4}\frac{2x-2 \alpha_k }{x^2-2 \alpha_k x+1}+\frac{1}{2}\frac{\alpha_k +1}{x^2-2 \alpha_k x+1}.
\end{equation*}
We first focus on the first three terms. Equation~(\ref{eq:FR2}) allows us to write
\begin{eqnarray*}
\sum_{k=1}^{n-1} \frac{1}{x}-\frac{1}{2}\frac{1}{x-1}-\frac{1}{4}\frac{2x-2 \alpha_k }{x^2-2 \alpha_k x+1} &=& \frac{n-1}{x} - \frac{n-1}{2(x-1)}-\frac{1}{2} \left(\frac{nx^{n-1}}{x^n-1}-\frac{1}{x-1}\right)\\
&=& \frac{n-1}{x} - \frac{n-2}{2(x-1)}-\frac{1}{2} \frac{nx^{n-1}}{x^n-1},
\end{eqnarray*}
and so
\begin{eqnarray*}
\frac{1}{n}\int_c^{\infty }\sum_{k=1}^{n-1} \frac{1}{x}-\frac{1}{2}\frac{1}{x-1}-\frac{1}{4}\frac{2x-2\alpha_k}{x^2-2\alpha_k x+1} dx &=& \frac{1}{2n} \left[\log\left(\frac{z^{2(n-1)}}{(z-1)^{n-2}(z^n-1)}\right)\right]_c^{\infty} \\
&=& \frac{n-2}{2n}\log(c-1)+\frac{1}{2n} \log(c^n-1)-2\frac{n-1}{2n}\log(c) \\
&=& \frac{n-1}{2n} \log(\ve)+\frac{\log(n)}{2n}+o(1).
\end{eqnarray*}
We also compute, with Lemma~\ref{lem:ArctanInt},
\begin{eqnarray*}
\frac{1}{n}\int_c^{\infty }\sum_{k=1}^{n-1} \frac{1}{2}\frac{\alpha_k +1}{x^2-2\alpha_k x+1} dx &=& \frac{1}{2n} \sum_{k=1}^{n-1}\left(\cos\left(\frac{2k\pi}{n}\right)+1\right) \int_c^\infty \frac{dx}{x^2-2\cos\left(\frac{2k\pi}{n}\right) x+1} \\
&=& \frac{1}{2n} \sum_{k=1}^{n-1} \frac{\cos\left(\frac{2k\pi}{n}\right)+1}{\sin\left(\frac{2k\pi}{n}\right)}\left( \frac{\pi}{2}-\frac{k \pi}{n}\right) \\
&=& \frac{1}{2n} \sum_{k=1}^{n-1} \cot\left(\frac{k \pi}{n}\right)\left( \frac{\pi}{2}-\frac{k \pi}{n}\right). \label{eq:In2}
\end{eqnarray*}

\end{proof}

The statement of Lemma~\ref{prop:ImnCoprime} and Lemma~\ref{prop:I1n} can be simplified by using the following
\begin{lemma} \label{lem:cotanzero}
If $m,n \ge 1$ are coprime, then $\d \sum_{k=1}^{n-1} \cot\left(\frac{m k \pi}{n}\right)=0$.
\end{lemma}

\begin{proof}
As $m,n$ are coprime, we know that the map $k \mapsto mk (n)$, i.e. the multiplication by $m$ modulo $n$, is one-to-one from $\{ 1,\ldots, n-1 \}$ onto itself, so 
\begin{equation*}
    \sum_{k=1}^{n-1} \cot\left(\frac{m k \pi}{n}\right)= \sum_{l=1}^{n-1} \cot\left(\frac{l \pi}{n}\right).
\end{equation*}
The change of indices $l \mapsto n-l$ and the formula $\cot(\pi-x)=-\cot(x)$ allow us to conclude.
\end{proof}

We can now complete the proof of Theorem \ref{th:Vasyunin}.
For $n \ge 1$, we set $E_n(t) = \frac{1}{nt}-\frac{1}{e^{nt}-1}$. We have 
\begin{equation*}
     I(m,n) = \int_0^\infty \left(E_m(t)-\frac{1}{m}E_1(t)\right)\left(E_n(t)-\frac{1}{n}E_1(t)\right) dt,
\end{equation*}
and Lemma~\ref{prop:ImnCoprime} gives a formula for $I(m,n)$. Moreover, Lemma~\ref{prop:I1n} gives a formula for $J(n) = \int_0^\infty E_n(t)E_1(t) dt$. In particular $J(1)=2 \cal C$. We now notice that
\begin{equation*}
    mn \int_0^\infty E_n(t) E_m(t) dt = mn I(m,n)+mJ(n)+nJ(m)-mnJ(1).
\end{equation*}
Adding up everything and using Lemma~\ref{lem:cotanzero} conclude the proof.

\section{Theorem~\ref{th:Vasyunin} for ${\rm gcd}(m,n)\neq 1$}

If $m,n$ are not coprime, the statement of Theorem~\ref{th:Vasyunin} cannot be correct, because the term $\cot\left(\frac{mk\pi}{n}\right)$ is not defined if $\frac{mk}{n}$ is an integer. However, we have the following:

\begin{lemma}
 \label{prop:limcotan}
Given integers $a,p,q \ge 1$, we have 
\begin{equation*}
\lim_{\lambda \rightarrow \frac{p}{q}} \cot(a \lambda q \pi) a \lambda q \pi + \cot\left(a \frac{1}{\lambda} p \pi \right)a \frac{1}{\lambda} p \pi = 1.
\end{equation*}
\end{lemma}

\begin{proof}
This limit comes from the expansion $\cot(x) = \frac{1}{x}+o(1)$ when $x \rightarrow 0$. We set $h = \lambda q-p $. Then $\frac{1}{\lambda}p-q = -\frac{1}{\lambda} h$, so we have:
\begin{eqnarray*}
\cot(a \lambda q \pi) a \lambda q \pi + \cot\left(a \frac{1}{\lambda} p \pi \right)a \frac{1}{\lambda} p \pi &=& \cot(a h \pi) a \lambda q \pi + \cot\left(-a \frac{1}{\lambda} h \pi \right)a \frac{1}{\lambda} p \pi \\
&=& \frac{a \lambda q \pi}{a h \pi} + \frac{a \frac{1}{\lambda} p \pi}{-a \frac{1}{\lambda} h \pi} + o_{h \rightarrow 0}(1) \\
&=& 1+ o_{h \rightarrow 0}(1).
\end{eqnarray*}
\end{proof}

Lemma~\ref{prop:limcotan} is only used to explain the abusive notations below:

\begin{theorem}
The Vasyunin-type formula stated in Theorem~\ref{th:Vasyunin} is valid for any integers $m,n \geq 1$ with the following abuse of notations: if $k \in \{1,\ldots,n-1\}$ and $l \in \{1,\ldots,m-1\}$ are such that $k/l=n/m$, we set:
\begin{equation} \label{eq:cotanlimit}
    \cot\left(\frac{m k \pi}{n}\right) \frac{m k \pi}{n}+\cot\left(\frac{n l \pi}{m}\right) \frac{nl \pi}{m} = 1.
\end{equation}
\end{theorem}

\begin{proof}
We only need to consider the case where $m,n$ are not coprime. We set $r={\rm gcd}(m,n)$, $m=rp$ and $n=rq$. The following is obvious:
\begin{equation*}
  mn\int_0^\infty \left(\frac{1}{mt}-\frac{1}{e^{mt}-1}\right)\left(\frac{1}{nt}-\frac{1}{e^{nt}-1}\right) dt =  r pq   \int_0^\infty \left(\frac{1}{pt}-\frac{1}{e^{pt}-1}\right)\left(\frac{1}{qt}-\frac{1}{e^{qt}-1}\right) dt, 
\end{equation*}
\begin{equation*}
    \frac{m-n}{2}\log\left(\frac{n}{m}\right)+ C(n+m) = r \left(\frac{p-q}{2}\log\left(\frac{q}{p}\right)+ C(q+p)\right).
\end{equation*}
We then notice that there exist exactly $r-1$ couples $(k,l) \in \{1,\ldots,n-1\} \times \{1,\ldots,m-1\}$ such that $k/l=n/m$. Such couples are given by $(k,l)=(a q,a p)$ with $a \in \{1,\ldots,r-1\}$, and equation~\eqref{eq:cotanlimit} gives:
\begin{equation*}
    -\frac{1}{2}\sum_{a=1}^{r-1} \cot\left(\frac{m aq \pi}{n}\right) \frac{m aq \pi}{n}+\cot\left(\frac{n ap \pi}{m}\right) \frac{nap \pi}{m} = -\frac{r-1}{2}.
\end{equation*}
It remains to compute the sum $- \frac{m}{2} \sum_{k \in K} \cot\left(\frac{m k \pi}{n}\right) \frac{k \pi}{n}$ when $K$ is the set of $k \in \{1,\ldots,n-1\}$ which are not multiples of $q$. Any $k$ in this set is written in a unique way under the form $k=aq+i$, with $a \in \{0,\ldots,r-1\}$ and $i \in \{1,\ldots,p-1\}$, so we have 
\begin{eqnarray*}
    - \frac{m}{2} \sum_{k \in K} \cot\left(\frac{m k \pi}{n}\right) \frac{k \pi}{n} &=& - \frac{m}{2}  \sum_{i=1}^{p-1} \sum_{a=0}^{r-1} \cot\left(\frac{m (aq+i) \pi}{n}\right) \frac{(aq+i) \pi}{n} \\
    &=& - \frac{p}{2}  \sum_{i=1}^{p-1} \sum_{a=0}^{r-1} \cot\left(\frac{p i \pi}{q}\right) \frac{(aq+i) \pi}{q} \\
    &=& - \frac{p}{2}  \sum_{i=1}^{p-1}\cot\left(\frac{p i \pi}{q}\right) \left(\frac{ar(r-1) \pi }{2}+\frac{ri \pi}{q} \right) \\
    &=& -r \frac{p}{2}  \sum_{i=1}^{p-1}\cot\left(\frac{p i \pi}{q}\right)\frac{i \pi}{q}.
\end{eqnarray*}
We used the $\pi$-periodicity of the cotangent function and Lemma~\ref{lem:cotanzero}. Similarly, we have
\begin{equation*}
    - \frac{n}{2} \sum_{l \in L} \cot\left(\frac{n l \pi}{m}\right) \frac{l \pi}{m} = -r \frac{q}{2}  \sum_{j=1}^{q-1}\cot\left(\frac{q j \pi}{p}\right)\frac{j \pi}{p},
\end{equation*}
where $L$ is the set of $l \in \{1,\ldots,m-1\}$ which are not multiples of $p$. We finally obtain that
\begin{equation*}
    \frac{m-n}{2}\log\left(\frac{n}{m}\right)+ C(n+m) -\frac{1}{2} - \frac{m}{2} \sum_{k=1}^{n-1} \cot\left(\frac{m k \pi}{n}\right) \frac{k \pi}{n}- \frac{n}{2} \sum_{l=1}^{m-1} \cot\left(\frac{n l \pi}{m}\right) \frac{l \pi}{m}
\end{equation*}
\begin{equation*}
    = r\left(\frac{p-q}{2}\log\left(\frac{q}{p}\right)+ C(q+p)-\frac{1}{2}-\ \frac{p}{2}  \sum_{i=1}^{p-1}\cot\left(\frac{p i \pi}{q}\right)\frac{i \pi}{q}-\frac{q}{2}  \sum_{j=1}^{q-1}\cot\left(\frac{q j \pi}{p}\right)\frac{j \pi}{p}\right)
\end{equation*}
\begin{equation*}
    =  r pq   \int_0^\infty \left(\frac{1}{pt}-\frac{1}{e^{pt}-1}\right)\left(\frac{1}{qt}-\frac{1}{e^{qt}-1}\right) dt = mn\int_0^\infty \left(\frac{1}{mt}-\frac{1}{e^{mt}-1}\right)\left(\frac{1}{nt}-\frac{1}{e^{nt}-1}\right) dt,
\end{equation*}

which is the general Vasyunin-type formula we wanted.
\end{proof}

\section*{Acknowledgement} The authors warmly thank Michel Balazard for communicating to them his proof of Proposition \ref{bal18}, and also thank Sandro Bettin for pointing out the reference \cite{ABB17}.

\end{document}